\documentclass{amsart}
\usepackage[english]{babel}
 
\usepackage{esint}
\usepackage{amsfonts,amssymb,amsmath}
\usepackage{setspace}
\usepackage{centernot}
\usepackage{mathrsfs} 
\usepackage{graphicx}

\newcommand{\erre} {{\mathbb {R}}}

\def\erren{{\erre^{ {N} }}}
\def\erreu{{\erre^{ {N+1} }}}
\def\effe{\mathscr{F}}
\def\av{``}
\def\cv{''}
\def\inn{\mbox{ in }}

\def\andd{ \quad\mbox{ and } \quad }

\def\pim{\mathcal{P}_m}
\newcommand{\tende}{\rightarrow}
\newcommand{\ttende}{\longrightarrow}
\newcommand{\enne} {\mathbb{N}}
\newcommand{\frecciaf} {\longmapsto}

\newcommand{\meno} {\smallsetminus}

\def\C {{\mathscr{C}}}
\def\subc {\underline{\mathscr{C}}(\Omega)}
\def\supc {\overline{\mathscr{C}}(\Omega)}
\newtheorem{theorem}{Theorem}[section]

\newtheorem{corollary}[theorem]{Corollary}

\theoremstyle{remark}
\newtheorem{remark}[theorem]{Remark}
\theoremstyle{definition}

\theoremstyle{remark}

\numberwithin{equation}{section}

\title[On the Perron solution of the caloric Dirichlet problem]{On the Perron solution \\of the caloric Dirichlet problem: \\ an elementary approach}

\author{Alessia E. Kogoj}
\address{Dipartimento di Scienze Pure e Applicate (DiSPeA)\\ 
				 Universit\`{a} degli Studi di Urbino Carlo Bo\\
				 Piazza della Repubblica 13, 61029 Urbino (PU), Italy.}
\email{alessia.kogoj@uniurb.it}

\author{Ermanno Lanconelli}
\address{Dipartimento di Matematica\\ 
				 Alma Mater Studiorum Universit\`{a} di Bologna\\
				 Piazza di Porta San Donato 5, 40126 Bologna, Italy.}
\email{ermanno.lanconelli@uniurb.it}


\begin{document} 

 \dedicatory{\flushright{ \av\  I like easy tricks.}\cv \\ \rm Louis Nirenberg }

\begin{abstract} By an easy trick taken from caloric polynomial theory we construct a family $\mathscr{B}$ of {\it almost regular} domains for the caloric Dirichlet problem. $\mathscr{B}$ is a basis of the Euclidean topology. This allows to build, with a basically elementary procedure, the Perron solution to the caloric  Dirichlet problem on every bounded domain.  
\end{abstract}
  
\maketitle

\section{Introduction and key theorem}
As it is very well known today, the construction of the Perron solution to the Dirichlet problem for the heat equation on a general bounded domain only rests upon three basic principles: the {\it caloric maximum principle}; the {\it convergence principle}, i.e., the closure of the sheaf of the caloric functions with respect to the uniform convergence; the  {\it solvability principle}, i.e., the solvability of the caloric first boundary value problem on the open sets of a basis of the Euclidean topology.

The first one of this principles is very elementary, and the second one is a simple consequence of some good properties of the Gauss--Weierstrass kernel, the fundamental solution of the heat equation. On the contrary, the proof of the solvability principle has been always considered in literature a difficult task, requiring Volterra integral equation theory, or double layer potential method, or an involved procedure based on a reflection principle (see e.g., respectively, Bauer \cite[Chapter 1, Section 2]{bauer}, 
Watson \cite[Chapter 2, Section 2.2]{W}, Constantinescu and Cornea \cite[Chapter 3, Section 3.3]{CC}).

The aim of this note is to draw attention to an easy trick - borrowed from caloric polynomial theory - allowing to construct with a very elementary procedure, a basis of open sets on which the first boundary value problem for the heat equation is solvable. The crucial point of this procedure is Theorem \ref{main} below, in which $w$ denotes the polynomial $$w(z)=w(x,t):= t -|x|^2.$$
Here and in what follows $$z=(x,t), \quad x\in \erren,\quad  t\in \erre,$$ denotes the point of $\erreu$; $|x|$ is the Euclidean norm of $x$.

We will use the notation $H$ to denote the {\it heat operator} 
$$H:=\varDelta -\partial_t,$$
where $\varDelta:=\sum_{j=1}^N \partial^2_{x_j}$ is the Laplacian in $\erren$. We call {\it caloric} the smooth functions solutions to $Hu=0$.  If $\Omega\subseteq\erreu$ is open, we will denote by $\C(\Omega)$ the linear space of the caloric functions in $\Omega$. 

To our aim it is convenient to fix some more notations. If $\alpha=(\alpha_1,\ldots, \alpha_{N},\alpha_{N+1}) $ is a multi-index with non negative integer components, we let $$|\alpha|_c = \mbox{\it caloric hight of }\alpha:=\alpha_1 + \ldots + \alpha_{N}+ 2 \alpha_{N+1}$$.

A polynomial in $\erreu$ is a function of the kind 

$$p(z)=\sum_{|\alpha|_c\le m} a_\alpha z^\alpha,\quad a_\alpha \in \erre \mbox{\ for every }\alpha,$$
where $m\in\mathbb{Z},\  m\geq0$. In this case we say that $p$ has caloric degree $\le m$; we will say that  $p$ has caloric degree equal to $m$ if $\sum_{|\alpha|_c= m} a_\alpha z^\alpha$ is not identically zero. 

Here is the key theorem of our note.

\begin{theorem}\label{main} Le $p$ be a polynomial in $\erreu$. Then, there exists a unique polynomial $q$ in $\erreu$ such that  \begin{eqnarray}\label{maineq} H(wq)=-Hp\end{eqnarray}
\end{theorem} 
\begin{proof} Let us denote by $m$ the caloric degree of  $-Hp$ and by $\pim$ the linear space of the polynomials in $\erreu$ having caloric degree less than or equal to $m$. Since $w$ has caloric degree two, then  $H(wq)\in\pim$ if $q\in\pim$; therefore 
$$q\frecciaf T(q):=H(wq)$$
maps $\pim$ in $\pim$. To prove that \eqref{maineq} has a unique solution $q$ it is (more than) enough to show that $T$ is surgetive and injective.  To this end, since $\pim$ is a linear space of finite dimension and $T$ linearly maps $\pim$ in $\pim$, we only have to prove that $T$ is injective.  Let $q\in\pim$ be such that $T(q)=0$. Then, $$u:=wq$$ is caloric in $\erreu$ hence, in particular, in the open region 

\begin{eqnarray}\label{region} P:=\{(x,t)\in\erreu \ :\ t > |x|^2 \}. \end{eqnarray}
Moreover, since $w=0$ on 

\begin{eqnarray}\label{bordo} \partial P=\{(x,t)\in\erreu \ :\ t = |x|^2 \}, \end{eqnarray}
$u|_{\partial P}=wq|_{\partial P}=0.$ Then, by the caloric maximum principle (see e.g. \cite[Theorem 8.2]{BB}) $u=0$ in $P$. Since $w>0$ in $P$, this implies $q=0$ in $P$, hence in $\erreu$. We have so proved that $q=0$  if $T(q)=0,$ that is the injectivity of $T$, completing the proof.
\end{proof}

From the previous theorem one immediately obtains the following corollary, in which $\partial P$ is the closed set in \eqref{bordo}, boundary of the region $P$ in \eqref{region}.

\begin{corollary}\label{corollario} Let $p$ be a polynomial in $\erreu$. Then, there exists a unique polynomial $u_p$ in $\erreu$ such that 
\begin{equation} \label{unotre}
\begin{cases}
 H u_p = 0\  \mbox{ in } \erreu,   \\
  u_p = p  \ \mbox{ on } \partial P. \end{cases}
\end{equation}

\end{corollary}
\begin{proof} Let $q$ be a polynomial satisfying \eqref{maineq}. Then $u_p=wq+p$ solves \eqref{unotre}. Moreover, if $v$ is any polynomial solving \eqref{unotre}, then $v -u_p$ is caloric in $\erreu$ - hence in $P$ - and $v - u_p=0$ on $\partial P$. The caloric maximum principle implies $v - u_p=0$ in $P$, hence in $\erreu$. Then $v = u_p,$ that is the uniqueness part of the Corollary.
\end{proof}

\section{A proof of the solvability principle:\\ a basis of $H$-almost regular domains}

Let $z_0=(x_0,t_0) \in \erreu$ and let $r>0$. We call:
\begin{itemize}
\item[$(i)$] {\it caloric bowl} of {\it bottom} $z_0$ and  {\it opening} $r$ the open set 
$$B(z_0,r): = \{ (x,t)\in\erreu\ :\ |x-x_0|^2< t-t_0<r^2\};$$
\item[$(ii)$] {\it normal} or {\it caloric boundary} of $B(z_0,r)$ the subset of $\partial B(z_0,r)$
$$\partial_n B(z_0,r): = \{ (x,t)\in\erreu\ :\ |x-x_0|^2= t-t_0, \ 0\le t-t_0\le r^2\}.$$

\end{itemize}
\vspace{0.8cm}
\begin{center}
\includegraphics[width=.8\textwidth]{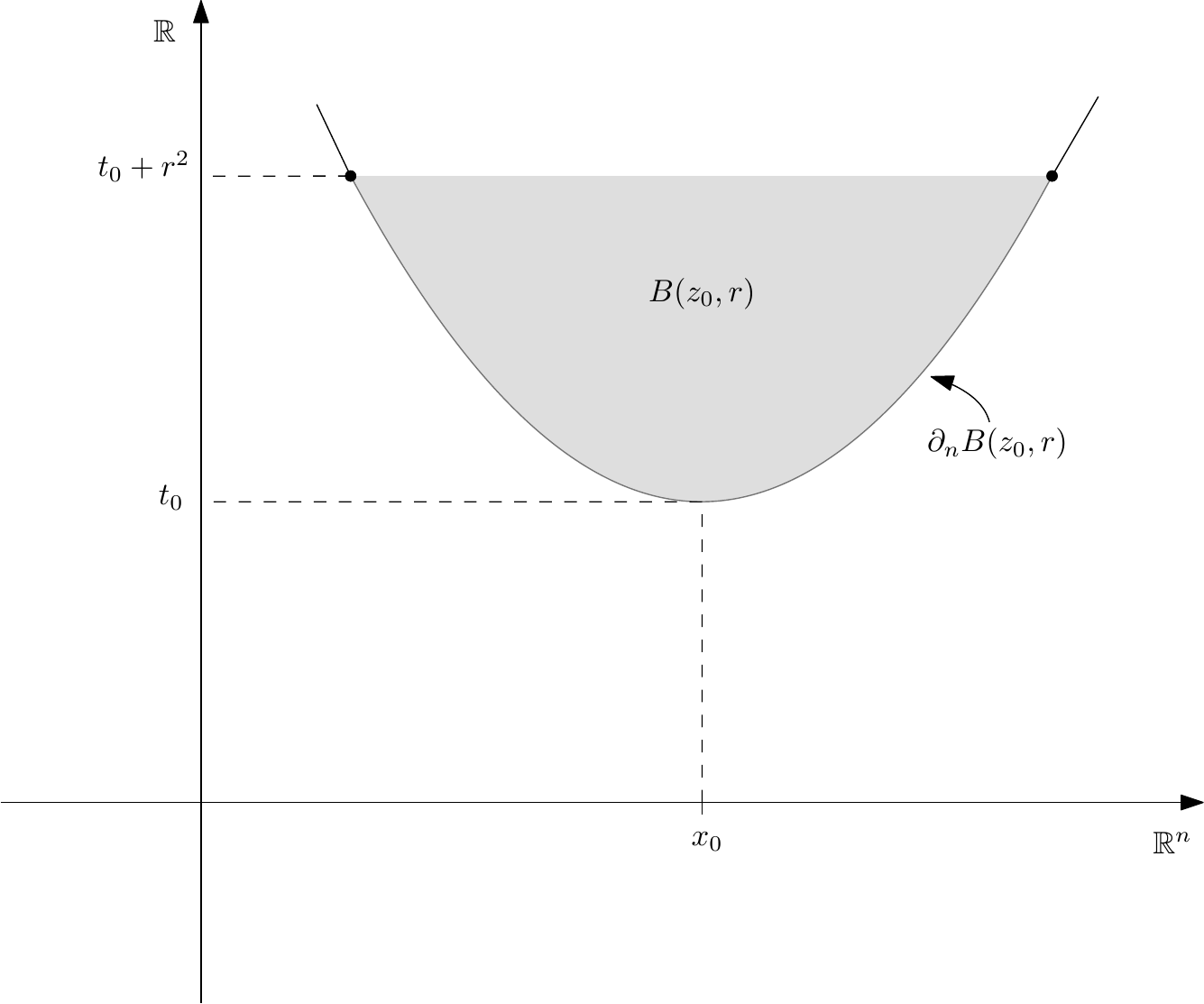}
\end{center}

Obviously, $$\mathscr{B}:=\{ B(z_0,r)\ : \ z_0\in\erreu,\ r>0 \}$$ is a basis of the Euclidean topology.

The aim of this section is to show that every caloric bowl is  {\it $H$-almost regular} in the sense of the following theorem.

\begin{theorem}\label{casamia}
Let $z_0\in\erreu$ and $r>0$ be arbitrarily fixed and let $B=B(z_0,r)$ be the caloric bowl of bottom $z_0$ and opening $r$. Then for every $\varphi \in C(\partial_n B,\erre)$ there exists a unique solution to the boundary value problem 
\begin{equation} \label{dueuno}
\begin{cases}
 H u = 0\  \mbox{ in } B,   \\
  u = \varphi \ \mbox{ on } \partial_n B. \end{cases}
\end{equation}

Precisely: there exists a unique function $u_\varphi^B$ caloric in $B$ and continuous up to $B\cup \partial_n B$ such that 
\begin{equation*} u^B_\varphi(z)=\varphi(z) \mbox{ for every } z\in \partial_n B.\end{equation*}
Moreover $u_\varphi^B \geq 0$ if $\varphi\geq 0$.
\end{theorem}
\begin{proof} The uniqueness of $u_\varphi^B$ and its positivity when $\varphi$ is positive is a direct consequence of the caloric maximum principle (see e.g. \cite[Theorem 8.2]{BB}). To prove the existence of $u_\varphi^B$ we may, and do, assume $z_0=(0,0),$ since $H$ is left translation invariant  and 
$$B(z_0,r)=z_0 + B(0,r),\quad 0\in\erreu.$$

Let $\varphi\in C(\partial_n B,\erre)$ and let $(p_k)_{k\in\enne}$ be a sequence of polynomials in $\erreu$ uniformly convergent 
to $\varphi$ on $\partial_n B.$ By using the notation of Corollary \ref{corollario}, we let 
$$u_k:=u_{p_k},\quad k\in \enne.$$

Then, $u_k$ is caloric in $B$ since it is caloric in $\erreu$. Moreover $u_k=p_k$ on $\partial P$ in \eqref{bordo}, hence on $\partial_n B.$ Therefore, by the caloric maximum principle, 

\begin{equation*} \max_ {\overline B} |u_k -u_h| =   \max_ {\partial_n B} |u_k -u_h|  = \max_ {\partial_n B} |p_k -p_h|  \ttende 0\mbox{ as }  k,h \ttende \infty.
\end{equation*}
From this, one gets the existence of a function $u\in C(\overline{B},\erre)$ which is caloric in $B$ - thanks to the convergence principle - and such that $$u|_{\partial_n B} = \lim_{k\tende \infty} u_k|_{\partial B} =  \lim_{k\tende \infty} p_k |_{\partial B} = \varphi.$$
Thus, the function $u$ is the requested function $u_\varphi^B.$
\end{proof}

\begin{remark} From the above proof it follows that the solution $u_\varphi$ of \eqref{casamia} actually is continuous up to $\overline{B}.$ 

\end{remark} 
For completeness reasons, and also to stress   its elementary character, in Appendix \ref{1} we will give a simple proof of the
{convergence principle}. In Appendix \ref{2} we sketch how to construct the caloric Perron solutions starting from Theorem \ref{casamia}.

\appendix 
\section{The convergence principle}\label{1}
\subsection*{$\bullet$ \it  The fundamental solution of $H$}
\mbox{}

 The Gauss--Weierstrass kernel, i.e., the function \begin{equation*} \Gamma:\erreu\ttende\erre,\quad \Gamma(x,t)=
\begin{cases}
 0\  \mbox{ if }  t\le 0,   \\
  (4\pi t)^{-\frac{N}{2}}\exp\left( -\frac{|x|^2}{4t}\right) \ \mbox{ if } t>0, \end{cases}
\end{equation*}
is the fundamental solution with pole at the origin of the heat operator in $\erreu$.
$\Gamma$ is smooth in $\erreu\meno\{(0,0)\}$ and locally summable in $\erreu$. Its crucial property is the following one:
\begin{equation}\label{a1} \varphi(z)= -\int_\erreu \Gamma(z-\zeta) H\varphi(\zeta)\ d\zeta \end{equation}
for every $z\in\erreu$ and for every $\varphi\in C_0^\infty( \erreu,\erre).$

\subsection*{\it $\bullet$ Caloric norm and caloric disks }
\mbox{}

If $z=(x,t)\in\erreu$ we let 
$$\| z\|= \mbox{{\it caloric norm} of $z$}:= (|x|^4+t^2)^{\frac{1}{4}}.$$
We call {\it caloric disk} of center $z_0 \in \erreu$ and radius $r>0$ the open set 
$$D(z_0,r):= \{ z\in\erreu \ : \ \| z -z_0\|<r \}.$$

\subsection*{\it $\bullet$ A representation formula for caloric functions }
\mbox{}

An easy consequence of property  \eqref{a1} is the following representation theorem 

\begin{theorem}\label{A1}  Let $\Omega\subseteq\erreu $ be open. For every caloric disk $D=D(z_0,r)$ such that 
$\overline{2D}:=\overline{D(z_0,2r)}\subseteq \Omega$, there exists a function $$(z,\zeta)\frecciaf K_D(z,\zeta)$$ of class 
$C^\infty$ in an open set containing $\overline{D\times 2D}$ such that 
$$u(z)=\int_{2D} K_D(z,\zeta )u(\zeta)\ d\zeta \mbox{\quad  for every } z\in D$$ 
and for every $u$ caloric in $\Omega.$ 

\end{theorem}
\begin{proof} Let $\psi\in C_0^\infty(2D,\erre)$ be such that $\psi\equiv 1$ in a neighbourhood of $\overline D$. Then $\psi u\in C_0^\infty (2D,\erre)$ and $u=\psi u$ in $D$. As a consequence, by \eqref{a1},  
\begin{eqnarray*} u(z)&=& -\int_{2D}  \Gamma(z-\zeta) H(\psi u )(\zeta)\ d\zeta \\ 
&=& -\int_D \Gamma(z-\zeta) \left( u(\zeta) H\psi (\zeta) + 2 \langle \nabla u(\zeta), \nabla \psi(\zeta)\rangle\right) \ d\zeta \mbox{\quad  for every } z\in D,\end{eqnarray*}
where $\nabla$ and $\langle \ , \  \rangle$ denote, respectively, the gradient and the inner product in $\erren.$ Integrating by parts the second summand at the last right hand side we find 

$$u(z)=\int_{2D} K_D(z,\zeta) u(\zeta)\ d\zeta,$$ 
where, denoting by $\varDelta$ the Laplacian with respect to the spatial variables, 
$$K_D(z,\zeta) =- \Gamma(z-\zeta) H\psi(\zeta) -2 \langle \nabla \Gamma(z-\zeta), \nabla \psi(\zeta)\rangle + \Gamma(z-\zeta)\varDelta \psi(\zeta).$$
Then $K_D$ is a smooth function in a neighbourhood of $D\times 2D$ and we are done.

\end{proof}

\subsection*{\it $\bullet$ The convergence principle}
\mbox{}
\\ As recalled in the Introduction the convergence principle is the statement of the following theorem.
\begin{theorem}\label{A1.2} Let $(u_k)$ be a sequence of caloric functions in an open set $\Omega\subseteq\erreu$. Suppose $(u_k)$ uniformly convergent to a function $u:\Omega\ttende\erre$ on every compact subset of $\Omega.$ Then,
$$u\in C^\infty(\Omega,\erre) \andd Hu=0 \inn \Omega.$$

\end{theorem}
\begin{proof} 
It is enough to show that $u$ is smooth and caloric in every caloric disk $D$ such that $\overline{2D}\subseteq\Omega.$ So, let $D$ be such a disk. Then, by Theorem \ref{A1}, 
\begin{equation}\label{a1.2} u_k(z)=\int_{2D} K_D(z,\zeta) u_k(\zeta)\ d\zeta  \mbox{\ for every } z\in D
\end{equation} 
and for every $k\in\enne$. Since $(u_k)$ is uniformly convergent on $\overline{2D}$, letting $k$ go to infinity in 
\eqref{a1.2}, we get 
\begin{equation*} u(z)=\int_{2D} K_D(z,\zeta) u(\zeta)\ d\zeta  \mbox{\ for every } z\in D.
\end{equation*} 
The smoothness of the kernel $K_D$ implies $u\in C^\infty(D,\erre).$ To show that $u$ is caloric in $D$ we argue as follows. Denoting $$H^*: 
=\varDelta + \partial_t$$ the formal adjoint of $H$, for every $\varphi\in C_0^\infty(D,\erre)$ we have 

\begin{eqnarray*} 
\int_D (Hu)\varphi\ dz = \int_D u\  H^* \varphi\ dz = \lim_{k\tende \infty} \int_D u_k\  H^* \varphi\ dz = \lim_{k\tende \infty} \int_D (H u_k)\  \varphi\ dz =0,
\end{eqnarray*} 
since $u_k$ is caloric for every $k\in\enne$. Hence,

$$\int_D (Hu)\ \varphi\ dz=0  \mbox{\quad  for every } \varphi\in C_0^\infty(\Omega,\erre).$$
This implies $Hu=0$ in $D$, completing the proof.

\end{proof} 

\subsection*{\it $\bullet$ A consequence of convergence principle}
\mbox{}

A family $\effe$ of real functions in an open set $\Omega\subseteq\erreu$ is said {\it up directed} if 
{\it for every $u,v\in\effe$ there exists $w\in\effe$ such that}
$$u\le w,\quad v\le w.$$
Then Theorem \ref{A1.2} and a Real Analysis lemma imply the following result.
\begin{theorem}\label{A1.3}Let $\effe$ be an up directed family of caloric functions in an open set $\Omega\subseteq\erreu.$
Let $u:\Omega\ttende ]-\infty,\infty]$,
$$u:=\sup\effe.$$
If $u$ is bounded above on every compact subset of $\Omega$, then 
$$u\in C^\infty(\Omega,\erre)\andd Hu=0\inn\Omega.$$

\end{theorem}
\begin{proof}  A Real Analysis Lemma (see e.g. \cite[Lemma 3.7.1]{AG}) implies the existence of a monotone increasing sequence $(u_k)$  of functions in $\effe$ such that 
$$\lim_{k\ttende\infty} u_k= u \mbox{\quad pointwise} \inn \Omega.$$
Now, we can argue as in the proof of Theorem \ref{A1.2}. Let $D$ be a caloric disk such that $\overline{2D}\subseteq\Omega.$ Then identity \eqref{a1.2} holds for every function $u_k$, $k\in\enne$. Since 
$$u_1\le u_k\andd \sup_{2D} u_k \le \sup_{2D} u<\infty,$$
by Lebesgue Dominate Convergence Theorem, letting $k$ go to infinity in  \eqref{a1.2}, one obtains 

$$u(z)=\int_{2D} K_D (z,\zeta)\ u(\zeta)\ d\zeta\qquad \forall z\in D.$$
The smoothness of the kernel $K$ implies the smoothness of $u$ in $D$ so that, since $(u_k)$ is increasing, by Dini Theorem $(u_k)$ converges uniformly on every compact subset of $D$. As a consequence, by Theorem \ref{A1}, $u$ is caloric in $D$.
This completes the proof of Theorem \ref{A1.3} since $D$ is any caloric disk such that $\overline{2D}\subseteq\Omega.$
\end{proof} 

\subsection*{\it $\bullet$ The operator $f\frecciaf h_f^B$}
\mbox{}

\indent Let $B$ be a caloric bowl. 
As an application of Theorem \ref{A1.3} we show how to extend the operator 
$$ C(\partial_n B,\erre)\ni\varphi\frecciaf u_\varphi^B\in \mathscr{C}(B),$$
defined in Theorem  \ref{casamia}, to the bounded above lover semicontinuous functions.

Let $B$ be a caloric bowl and let 
$$ f: \partial_n B \ttende ]-\infty,\infty[$$
be a bounded above lover semicontinuous function. Define 
$$\effe(B,f):= \{u^B_\varphi \ :\ \varphi\in C(\partial_n B,\erre),\ \varphi\le f \} $$
and $$h_f^B:=\sup \effe(B,f).$$
Obviously, if the function $f$ is continuous, then $h_f^B=u_f^B,$ so that $f\frecciaf h_f^B$ is an extension of $f\frecciaf u_f^B$.By using the caloric maximum principle, it is easy to show that
$\effe(B,f)$ is up directed and that 
$$h_f^B\le m\mbox{\quad if \ } m=\sup_{\partial_n B} f.$$
Then, by Theorem  \ref{A1.3},
$$h_f^B\in \C(B).$$

\section{The caloric Perron solution}\label{2}

\subsection*{\it $\bullet$ Mean Value Theorem for caloric functions}
\mbox{}

For every $z_0\in\erreu$ and every $r>0$ we let 
\begin{eqnarray*} \Omega_r(z_0)  &=& \mbox{Pini--Watson ball with pole at\ }  z_0 \mbox{\ and radius \ }  r \\
 &: =& \left\{ z\in\erreu \ : \ \Gamma(z_0-z) > (4\pi r)^{-\frac {N}{2}}\right\} 
\end{eqnarray*} 
and 
\begin{eqnarray*} W(z)=W(x,t)   &=& \mbox{Watson kernel \ }  \\
 &: =&\frac {1}{4} \left( \frac{|x|}{t}\right)^2. 
\end{eqnarray*} 
We also denote by $M_r(u)(z_0)$ the average operator 
$$M_r(u)(z_0)= \left( \frac{1}{4\pi r}\right)^{\frac {N}{2}} \int_{\Omega_r (z_0)} u(\zeta)\ W(z-\zeta)\ d\zeta.$$
Then, the following theorem holds

\begin{theorem} \label{A2.1} Let $\Omega\subseteq\erreu$ be open and let $u\in C(\Omega,\erre).$ The following 
statements are equivalent:
\begin{itemize}
\item[$(i)$] For every Pini--Watson ball $\Omega_r(z)$ with closure contained in $\Omega$ $$u(z)=M_r(u)(z).$$ 
\item[$(ii)$] For every $z\in\Omega$ there exists $r(z)>0$ such that 
$$u(z)=M_r(u)(z) \mbox{\qquad for \ } 0<r<r(z).$$ 
\item[$(iii)$]  $u\in C^\infty (\Omega,\erre)$ and 
$$Hu=0\inn \Omega.$$

\end{itemize}
\begin{proof} See Watson \cite[Chapter 1 and Chapter 2]{W}.
\end{proof} 
\end{theorem} 

\subsection*{\it  $\bullet$ Supercaloric and subcaloric functions}
\mbox{}
Let $\Omega\subseteq\erreu$ be open and let 
$$u:\Omega\ttende\erre$$
be a lower semicontinuous function. We say that $u$ is {\it supercaloric in $\Omega$} if for every $z\in\Omega$ there exists $r(z)>0$ such that 
$$u(z) \geq M_r(u)(z)\mbox{\qquad for \ } 0<r<r(z).$$  

We say that $u$ is {\it subcaloric in $\Omega$} if $-u$ is supercaloric. With  $\supc$ ($\subc$), we denote the family of the supercaloric (subcaloric) functions in $\Omega.$
It can be elementarily proved that a sufficiently smooth function $u$ is supercaloric (subcaloric) in an open set $\Omega$ if and only if 
$$Hu\le 0\inn\Omega\quad  (Hu\geq 0\inn\Omega).$$

\subsection*{\it $\bullet$ A caloric Perron-type regularitazion}
\mbox{}
To begin with, we fix some notation.  If $B$ is a caloric bowl we denote 
$$\hat B:= \overline{B}\meno \partial_n B.$$

Equivalently 

$$\hat B=B\cup \mathrm{top }(B), $$
where
$$\mathrm{top }(B):= \partial B\meno\partial_n B.$$

Let us consider a bounded above supercaloric function $u$ in an open set $\Omega\subseteq\erreu.$ If $B=B(z_0,r)$ is a caloric bowl such that $2B:=B(z_0,2r)\subseteq \Omega$, we define
$$u_B:\Omega\ttende\erre$$ as follows:
$$u_B(z)=u(z) \mbox{\quad if \ } z\notin \hat B,$$
and 
$$u_B(z)=h_f^{2B}(z)  \mbox{\quad if \ } z\in \hat B,$$
where $f=u|_{\partial_n 2B}.$

We want to explicitly remark that $u_B$ is {\it caloric} in $B$ and {\it continuous} up to $\hat B$. The function $u_B$ is what we call {\it caloric Perron-type regularization of $u$ in $B$.} It satisfies all the crucial properties of the classical harmonic Perron regularization. Precisely, the following theorem holds.

\begin{theorem} \label{A2.2} 
Let $u$ be a bounded above supercaloric function in an open set $\Omega\subseteq\erreu$ and let $B$ be a bowl such that $2B\subseteq\Omega.$ Then,

\begin{itemize}
\item[$(i)$] $u_B\in \supc$;

\item[$(ii)$] $u_B\le u$;

\item[$(iii)$]  $u_B$ is caloric in $B$ and continuous in $\hat B$;
\item[$(iv)$]  if $v\in\supc$ and $v\le u$ then $v_B\le u_B.$

\end{itemize}

\end{theorem} 

\begin{proof} The proof of this theorem follows basic standard lines in harmonic and caloric Potential Theory. It uses in a crucial way the properties of the operator $f\frecciaf u_f^{2B}$ and the minimum principle for supercaloric functions (for this principle 
we directly refer to Watson's monograph \cite[Theorem 3.11]{W}.
\end{proof} 
We close this appendix with the following point.
\subsection*{\it $\bullet$ The caloric Perron solution}
\mbox{}\\
Let $\Omega\subseteq\erreu$ be open and bounded and let $\varphi\in C(\partial\Omega, \erre).$ We let 

$$\overline {\mathcal{U}}^\Omega_\varphi:= \{ u\in\supc\ : \ u \mbox{\ bounded above,}\ \liminf_{x\tende y} u(x) \geq \varphi(y)\ \forall y\in\partial \Omega\},$$

and $$\overline{H}_\varphi^\Omega:=\inf \overline{\mathcal{U}}^\Omega_\varphi.$$

We also let 

$$\underline{H}_\varphi^\Omega:=-\overline{H}_{-\varphi}^\Omega.$$

From the quoted above minimum principle for supercaloric functions one easily gets 

$$m\le \underline{H}_\varphi^\Omega \le \overline {H}_\varphi^\Omega \le M,$$
where $$m=\min_{\partial \Omega} \varphi \andd M=\max_{\partial \Omega} \varphi .$$
Actually, a stronger result holds
\begin{theorem} \label{A2.3}  For every $\varphi \in C(\partial \Omega,\erre)$ one has 

\begin{itemize}
\item[$(i)$] $\overline {H}_\varphi^\Omega  \in \C(\Omega)$;

\item[$(ii)$] $\overline {H}_\varphi^\Omega=\underline {H}_\varphi^\Omega.$

\end{itemize}

\end{theorem} 
This is the caloric version of the celebrated Perron--Wiener Theorem for the harmonic functions. It can be proved with a standard procedure in which the caloric Perron-type regularization plays the crucial r\^ole.

\section*{Acknowledgment}
The first author  has been partially supported by the Gruppo Nazionale per l'Analisi Matematica, la Probabilit\`a e le
loro Applicazioni (GNAMPA) of the Istituto Nazionale di Alta Matematica (INdAM).

\bibliographystyle{alpha} 

\begin{thebibliography}{Wat12}

\bibitem[AG01]{AG}
D.~H. Armitage and S.~J. Gardiner.
\newblock {\em Classical potential theory}.
\newblock Springer Monographs in Mathematics. Springer-Verlag London, Ltd.,
  London, 2001.

\bibitem[Bau66]{bauer}
H.~Bauer.
\newblock {\em Harmonische {R}\"aume und ihre {P}otentialtheorie}.
\newblock Ausarbeitung einer im Sommersemester 1965 an der Universit\"at
  Hamburg gehaltenen Vorlesung. Lecture Notes in Mathematics, No. 22.
  Springer-Verlag, Berlin-New York, 1966.

\bibitem[BB19]{BB}
S.~Biagi and A.~Bonfiglioli.
\newblock {\em An introduction to the geometrical analysis of vector fields
  with applications to maximum principles and {L}ie groups}.
\newblock World Scientific Publishing Co. Pte. Ltd., Hackensack, NJ, 2019.

\bibitem[CC72]{CC}
C.~Constantinescu and A.~Cornea.
\newblock {\em Potential theory on harmonic spaces}.
\newblock Springer-Verlag, New York-Heidelberg, 1972.
\newblock With a preface by H. Bauer, Die Grundlehren der mathematischen
  Wissenschaften, Band 158.

\bibitem[Wat12]{W}
N.~A. Watson.
\newblock {\em Introduction to Heat Potential Theory}.
\newblock Mathematical Surveys and Monographs vol.182, Amer. Math. Soc.,
  Providence RI, 2012.

\end{thebibliography}

\end{document}